\documentclass{article}
\usepackage{graphicx} % Required for inserting images
\usepackage{amsfonts}
\usepackage{url}
\usepackage[english]{babel}

\newcommand{\rvline}{\hspace*{-\arraycolsep}\vline\hspace*{-\arraycolsep}}
\usepackage{amsthm}
\usepackage{pinlabel}
\theoremstyle{definition}
\newtheorem{definition}{Definition}[section]
\newtheorem{theorem}{Theorem}[section]
\newtheorem{proposition}{Proposition}[section]

\newtheorem{example}{Example}[section]
\title{Learning bridge numbers of knots}

\usepackage{booktabs}
\usepackage{amsmath}
\usepackage{algorithm}
\usepackage{algorithmic}
\usepackage{geometry}
\usepackage{multirow}

\usepackage[sort,nocompress]{cite}

\author{Puttipong Pongtanapaisan, Hanh Vo, Thieu Nguyen} 
\date{February 2024}

\usepackage{xcolor}   % For colored text

\usepackage{listings} % For code listings
\begin{document}

\maketitle
\begin{abstract}
This paper employs various computational techniques to determine the bridge numbers of both classical and virtual knots. For classical knots, there is no ambiguity of what the bridge number means. For virtual knots, there are multiple natural definitions of bridge number, and we demonstrate that the difference can be arbitrarily far apart. We then acquired two datasets, one for classical and one for virtual knots, each comprising over one million labeled data points. With the data, we conduct experiments to evaluate the effectiveness of common machine learning models in classifying knots based on their bridge numbers.
\end{abstract}

\section{Introduction}
Traditionally, knot theory is the study of embeddings of circles in 3-space. It can be used to model real-world entanglements such as vortex knots and linked vortex rings in water \cite{kleckner2013creation}. Kauffman then introduced a natural generalization of knot theory called virtual knot theory \cite{kauffman2021virtual}, which researchers have used to study knots in spaces more complicated than the 3-space, and knotted surfaces in 4-space \cite{kauffman2008invariants}. A function called \textit{invariant} can be defined on the set of (virtual) knots that measures how complicated each knot type is. A common way to define an invariant is to count a certain quantity obtainable from a presentation of a knot, and then minimize over all presentations of the knot type. An example of such an invariant is called the \textit{bridge number}.

In classical theory, there are various ways to define the same quantity called the bridge number of a knot. For instance, it can be defined as the minimum number of overbridges over all diagrams or the minimum number of local maxima over all reasonable embeddings. In contrast, these two quantities are no longer the same for virtual knots \cite{nakanishi2015two}. Thus, it makes sense to distinguish between the two versions of the bridge number as the \textit{first bridge number} and the \textit{second bridge number}. Examples where the second bridge number is one more than the first bridge number exist in the literature, but in this work, we also show that the difference between the two versions can be arbitrarily large.

Turning to computational aspects, the techniques developed in \cite{Blair_2020,nelson2006matrices,pongtanapaisan2019wirtinger} allow us to know the exact values of bridge numbers of over a million classical knots up to 16 crossings. 
For virtual knots, due to the lack of tabulated distinct virtual knots, we first modified existing Gauss codes representing classical knots to obtained over 100 millions virtual knots of 15 crossings. Although we do not prove rigorously that these modified codes represent distinct knots, we give some reasoning that most of them are.
Then we used the codes in \cite{Blair_2020, Paul} to compute their first bridge numbers, with an adaption in the part of computing Wirtinger number.
By doing so, we again were able to obtain the exact values of first bridge numbers of over a million virtual knots. 
For the second bridge number, we show that there is a lower bound coming from biquandles. We then used the code from \cite{nelson2006matrices} to compute the lower bound for the second bridge number for all reduced virtual knots up to 5 crossings.

Since one needs to consider infinitely many diagrams of a knot to compute the bridge number, there is no  algorithm so far that can compute the exact value of any version of the bridge number of a knot. 
Machine Learning has been successfully utilized to discover the relationships between knot invariants and predict their values (see \cite{davies2021advancing,hughes2020neural} and the references therein).  
For instance, the authors in \cite{hughes2020neural} used  
Neural Network in to investigate 
whether a knot can be realized as the closure of a quasipositive braid.

As mentioned above, since we have a sufficiently large labelled data set of knots with exact bridge numbers, we could use machine learning in the problem of classifying the bridge number. We experimented with classifying 3 and 4 bridge  classical knots of 16 crossings. Upon testing with several machine learning classifiers such as Neural network, KNN, SVN, Decision tree, we see that random forest is one of the promising methods. It consistently gives high performance, with accuracy $95.99\pm0.05\%$ and weighted F1-score $95.82\pm0.05\%$. As the bridge number computation is tied to the meridional rank conjecture, which is an unsolved problem 1.11 in \cite{kirby1997problems}, this might lead to some beneficial insights in future works.

\subsection*{Plan of the paper}
The paper is organized as follows. In Section \ref{sec:review}, we define mathematical terms relevant to the paper (e.g., the various formulations of the bridge number). We also discuss existing results and computational techniques in the literature \cite{blair2022coxeter,nelson2006matrices}. In Section \ref{section_biquandles}, we provide a lower bound of the second version of the bridge number of virtual knots using biquandles. We use Python code to compute the lower bound for several virtual knots and give a family where the two formulations of bridge numbers are arbitrarily far apart. In Section \ref{Sec:Experiments}, we compute the first version of bridge number for a large number of classical and virtual knots. We also experiment with machine learning models, and train the machine to predict unknown bridge numbers. In Section \ref{sec:Conclusions}, we conclude the paper and discuss future works.

\section{Preliminaries}\label{sec:review}
\subsection{Knots and their presentations}
A \textit{knot} is an embedding of a circle in 3-space $\mathbb{R}^3$. Two knots are considered equivalent if there is a homeomorphism of $\mathbb{R}^3$ to itself taking one knot to the other. Roughly speaking, this means one can be smoothly
deformed into the other. We also call an equivalence class of knots a \textit{knot type}.

In computational knot theory, it is customary to represent a knot by a sequence of integers.
\begin{definition}
    A \textit{Gauss code} is a sequence $\{a_i\}_{i=1}^{2n}$, where each term $a_i$ comes from the set $\{\pm 1,\pm 2,\cdots, \pm n\}$ appearing only once.
\end{definition}

 Now, we describe how to obtain a Gauss code from a knot diagram. A knot admits a diagram in the plane, where two small segments are removed at each double point to indicate whether the one strand goes over or under the other. Pick a point on the diagram to be the starting point, and label the $n$ crossings with numbers $1,2,\cdots, n$, and pick a direction. As one travels along the selected direction from the starting point, record the crossing labels in order. There is an additional rule: if we encounter a crossing $i$ as an undercrossing, we record the label $-i$. For instance, the knot diagram in Figure \ref{fig:gauss} has an associated Gauss code $\{-1,2,-3,1,-2,3\}$.
\begin{figure}[h]\label{fig:gauss}
   \centering
\includegraphics[width=4cm]{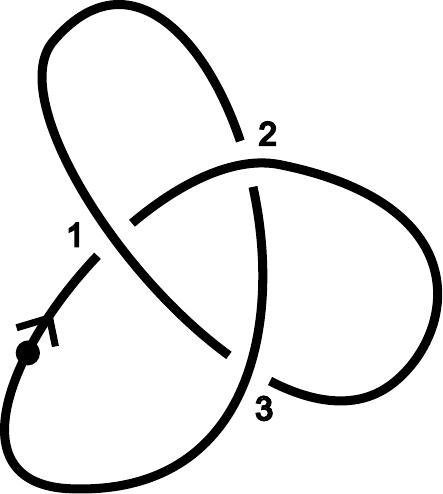}
\caption{A diagram with Gauss code $\{-1,2,-3,1,-2,3\}$.}
\end{figure}

Of course, multiple Gauss codes can correspond to the same knot type. Furthermore, not every Gauss code corresponds to a knot. This motivates Kauffman to enlarge the knot universe to include \textit{virtual knots} \cite{kauffman2021virtual}. This is analogous to including $\sqrt{-1}$ so that any polynomial equation has a complex solution. For instance, consider the Gauss code $\{ 1,-2,-1,2\}$ (recall that since we are representing a circle, we read the code circularly and 2 must connect to 1). If one tries to construct a knot diagram corresponding to this code, it will become apparent that the diagram cannot be closed up (see Figure \ref{fig:nonplanar}).

\begin{figure}[h]\label{fig:nonplanar}
   \centering
\includegraphics[width=9cm]{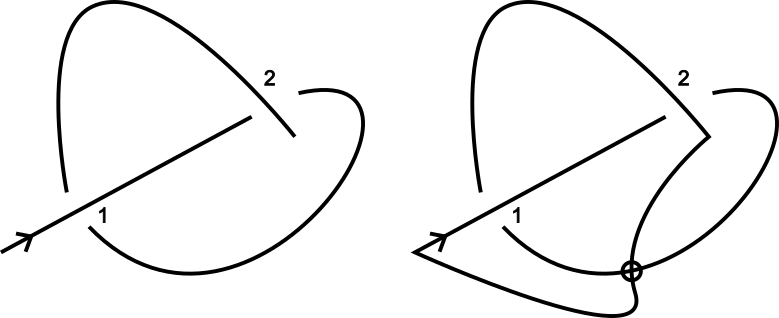}
\caption{A nonplanar Gauss code}
\end{figure}

Formally, virtual knots are in one-to-one correspondence with oriented Gauss codes up to Reidemeister moves, where the signs of the crossings are also taken into consideration. For more details, the readers are encouraged to consult \cite{kauffman2021virtual}. For our purpose of studying the bridge number, however, it suffices to think of virtual knots as Gauss codes defined as above and ignore the crossing signs. We define what a virtual knot is here anyway for completeness. 

\begin{definition}
    An \textit{oriented Gauss code} is a decorated Gauss code, where the entries $i$ and $-i$ may be decorated with a bar symbol.
\end{definition}

Whether a number in a Gauss code has a bar symbol depends on the two possible orientations near a crossing (see Figure \ref{fig:biqrules}). For instance, Figure \ref{fig:nonplanar} has Gauss code $\{ 1,-2,-1,2\}$, but its oriented Gauss code is still $\{ 1,-2,-1,2\}$ (i.e., all crossings are of the type shown on the right of Figure \ref{fig:biqrules}). On the other hand, the oriented Gauss code for Figure \ref{fig:gauss} is $\{-\overline{1},\overline{2},-\overline{3},\overline{1},-\overline{2},\overline{3}\}$.
\begin{definition}\label{def:orientedgauss}
    Let $S_i$ denote sub-sequences contained in a Gauss code, and $a,b,c$ denote positive numbers. A \textit{virtual knot} is an equivalence class of decorated Gauss codes up to cyclic permutation and the following moves:
\begin{enumerate}
    \item $S_1 \leftrightarrow a,-a,S_1.$ 
    \item $S_1S_2\leftrightarrow a,\overline{b},S_1,-a,-\overline{b},S_2$.
    \item $a,b,S_1,-a,c,S_2,-b,-c,S_3 \leftrightarrow b,a,S_1,c,-a,S_2,-c,-b,S_3$
\end{enumerate}.

\end{definition}

Roughly, moves $i$ for $i=1,2,3$ in Definition \ref{def:orientedgauss}, correspond to moves in the knot diagrams that involve $i$ crossings. Move 1 is actually depicted in Figure \ref{fig:gaussr1}.
Of course, within each of the moves above there are a few versions. For instance, a variation of the first move that is also allowed is $S_1 \leftrightarrow -a,a,S_1.$ The readers can work out all the versions as an exercise or consult the existing literature. 

The point is that if these moves are performed on ``realizable" Gauss codes (i.e., the ones corresponding to classical knots), then we are describing knots in 3-space up to the geometric equivalence relation described in the beginning. However, if we consider all possible Gauss codes, then we get virtual knot theory. Therefore, one can technically study virtual knot theory without drawing any pictures at all, but that would not be very fun.
\subsection{The bridge number: a geometric complexity}
Since any (virtual) knot can be presented as a Gauss code, we can use some features of the codes to measure the complexity of a knot type. The bridge number is one of the most useful invariants in knot theory, and is defined as follows. We begin by defining the bridge number of a Gauss code. We call a consecutive string of numbers starting an ending at negative integers a \textit{strand}. A strand containing at least one positive integer is called an \textit{overbridge}. Here, we treat the code cyclically as the code represents a circle in 3-space. For instance, consider the Gauss code $\{1,-4,-3,2,4,-1-2,3\}$. Then, there are two overbridges [2,4], and [3,1]. Here in cyclic ordering, the number 3 wraps around so that it is right next to the number 1. Thus, the bridge number of this Gauss code is 2.  

\begin{definition}
 The minimum number of overbridges over all Gauss codes representing the same knot type is called the \textit{bridge number} of a knot type $K$. We denote this by $b_1(K).$
\end{definition}
As mentioned before, each Gauss code has a knot diagram realization. This means that the definitions such as strands, overbridges, and bridge numbers can also be rephrased in terms of pictures and diagrams. For instance, a strand is an arc in the diagram from an undercrossing to another undercrossing.

We remark that it can be difficult to compute the exact bridge number since there are infinitely many Gauss codes representing the same knot type, and we have to search for the minimum. 

\begin{example}
    The Gauss code in Figure \ref{fig:gaussr1} is longer in length, but it represents the same knot type as the Gauss code from Figure \ref{fig:gauss}. We can see that the corresponding knots are equivalent by a slight perturbation. Another thing to notice is that $\{-1,2,-4,4,-3,1,-2,3\}$ has 4 overbridges, while the diagram in Figure \ref{fig:gauss} has 3 overbridges.  
    Say this diagram of interest represents the knot type $K$. One can prove that the actual bridge number of $K$ is 2. That is, there is a different Gauss code for $K$ with only 2 overbridges.
\end{example}

\begin{figure}[h]\label{fig:gaussr1}
   \centering
\includegraphics[width=4cm]{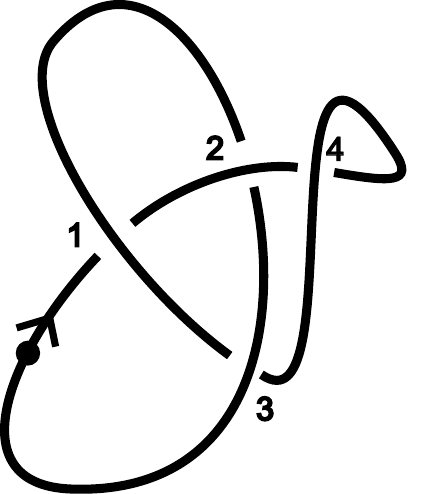}
\caption{A diagram with Gauss code $\{-1,2,-4,4,-3,1,-2,3\}$. Note that move 1 from Definition \ref{def:orientedgauss} was performed.}
\end{figure}

There is another way to define the bridge number for classical knots. Since a knot diagram is a drawing of a circle, we can perturb it so that the diagram is nice with respect to the standard height function of the plane $h:\mathbb{R}^2\rightarrow\mathbb{R}^2$ in the sense that the number of local maxima is the same as the number of local minima. For classical knots, it turns out that the minimum number of local maxima over all diagrams for a classical knot type $b_2(K)$ is exactly the same as $b_1(K)$ (See discussions in \cite{nakanishi2015two}). Now, if $K$ is a virtual knot type, then examples were found (See \cite{hirasawa2011bridge}, for instance) demonstrating that there is a virtual knot type $K$, where $b_1(K)=1$ and $b_2(K)=2.$

The authors in \cite{nakanishi2015two} referred to  $b_1(K)$ as the \textit{first bridge index} and $b_2(K)$ as the \textit{second bridge index}. We will refer to these quantities as the \textit{first bridge number} and the \textit{second bridge number}, respectively. A goal of this paper is to demonstrate that $b_1$ and $b_2$ can be very different. To show this, we employed the technique due to the first author, where $b_1$ was defined in an algorithmically friendly way.

Let $D$ be a diagram where $k$ strands of $D$ are colored. Consider a collection of crossings $\{c_1,...,c_m\}$, where each crossing $c_i$ has the property that an overstrand is colored, and the incoming strand is colored. Then, a coloring move on $\{c_1,...,c_m\}$ is an assignment of $m$ more colorings to $D$, resulting in a new diagram $D'$, where the outgoing strands at $\{c_1,...,c_m\}$ receive colors.

This process is demonstrated in Figure \ref{fig:demonstrate}. In the figure, there are six stages. In stage one, we have a diagram $D_1$, where one strand receives a color. The initial colored strands are also called the \textit{seeds}. Two coloring moves are then applied to two crossings of $D_1$ yielding $D_2$. The crossings are marked by transparent grey circles. Two coloring moves then take $D_2$ to $D_3.$ To go from each $D_j$ to $D_{j+1}$ where $j=3,4,5,$ we only see one crossing that satisfies the condition of the coloring move. Thus, there is one coloring move for each $j$ until we arrive at the end $D_6,$ where every strand is colored.

Generalizing the work of \cite{Blair_2020}, the first author showed in \cite{pongtanapaisan2019wirtinger} that if one can color a diagram $D$ with $s$ initial seeds and extend the coloring using only the coloring moves to the entire diagram, then the bridge number $b_1(K)$ is at most $s.$ Therefore, while the diagram (which readers can convert to Gauss code) in Figure \ref{fig:demonstrate} may seem to contain many overbridges, it turns out that $b_1(K) = 1$.

\begin{figure}[h]\label{fig:demonstrate}
   \centering
\includegraphics[width=8cm]{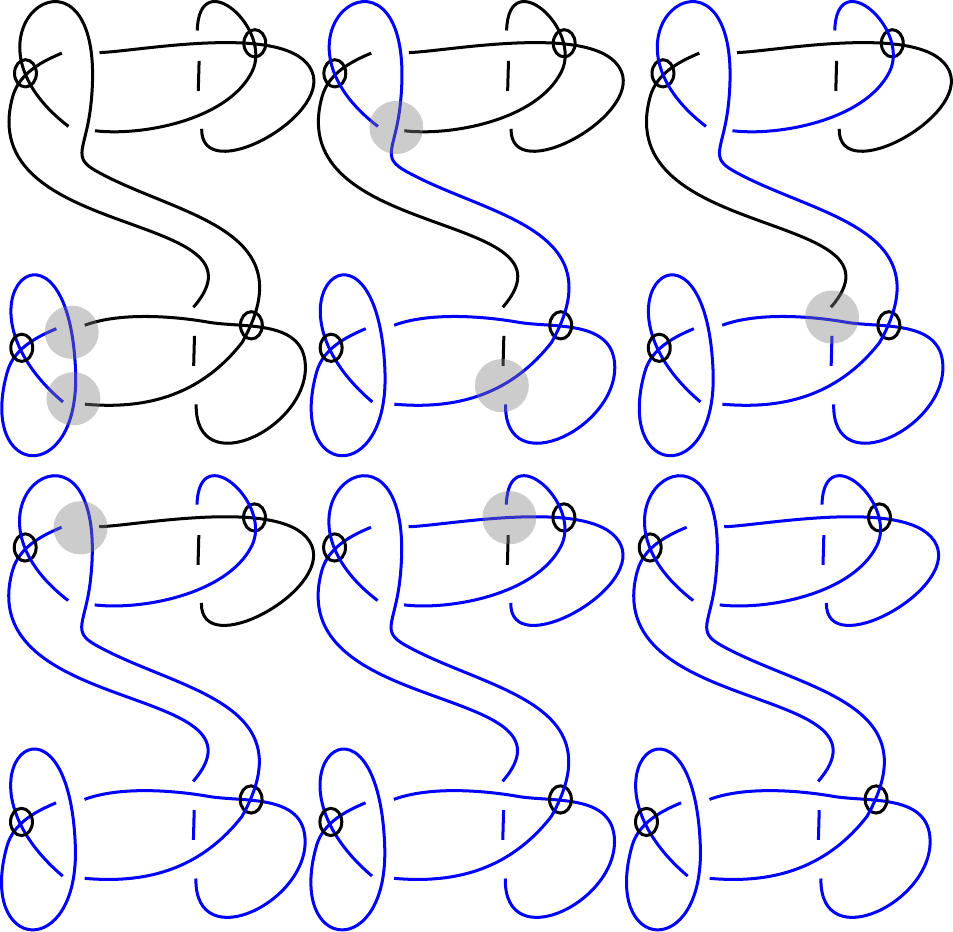}
\caption{Demonstrating the coloring method, showing that $b_1(K)=1.$}
\end{figure}

\subsection{The meridional rank: an algebraic complexity}
In the previous subsection, we defined a geometric measure that can be used to deduce how complex a knot is. In this subsection, we discuss a closely related algebraic invariant, which is conjectured to be the same as the bridge number. This way, our work can be thought of as predicting this algebraic quantity as well.

Consider the knot complement $X=\mathbb{R}^3\backslash K$. In other word, this is the space ``around" the knot. We can form a group whose elements are based loops in $X.$ A \textit{meridian} is an element in this group that is freely homotopic to a boundary of a disk embeded in 3-space intersecting the knot once. Roughly, a meridian can be continuously deformed to be a very simple loop. The \textit{meridional rank} is the smallest number of meridians needed to generate the group.

Even if the minimal size over all generating sets of a group is known, it is still an interesting question to ask for the minimum number of generating sets made up of ``standard" elements. For instance, the mapping class groups can always be generated with two elements if we do not restrict the types of elements \cite{wajnryb1996mapping}. However, the number of generators by standard Dehn twists can be arbitrarily large \cite{humphries2006generators}. This article investigates this problem for knot groups from a machine learning perspective. 

\section{Estimation from biquandles: difference between $b_1$ and $b_2.$}\label{section_biquandles}

In \cite{blair2022coxeter}, the authors were able to compute lower bounds of the bridge number by searching for maps from knot groups to Coxeter groups. Their method was very effective, but does not work for all knots. In this section, we review and generalize the lower bound coming from quandle theory.

A \textit{biquandle} $X$ is a set equipped with two binary operations $\overline{\triangleright}, \underline{\triangleright}$ satisfying the following properties.
\begin{enumerate}
    \item[(i)] For all $x\in X$, $x \overline{\triangleright} x=x\underline{\triangleright} x$,
    \item[(ii)]  The maps $\alpha_y, \beta_y : X \rightarrow X$ and $S : X \times X \rightarrow X \times X$ defined by $\alpha_y(x) = x \overline{\triangleright} y, \beta_y(x) = x \underline{\triangleright}  y$ and
$S(x, y) = (y \overline{\triangleright}x, x \underline{\triangleright} y)$ are invertible,
\item[(iii)] For all $x,y,z\in X$, $(x\overline{\triangleright} y)\overline{\triangleright} (z\overline{\triangleright}  y) = (x\overline{\triangleright} z)\overline{\triangleright}(y\underline{\triangleright} z)$ \\ $(x\underline{\triangleright} y)\overline{\triangleright} (z\underline{\triangleright}  y) = (x\overline{\triangleright}z)\underline{\triangleright} (y\overline{\triangleright} z)$ \\ $(x\underline{\triangleright}y)\underline{\triangleright}(z\underline{\triangleright} y) = (x\underline{\triangleright}z)\underline{\triangleright}(y\overline{\triangleright}  z)$,
\end{enumerate}

For practical computations, a biquandle can be encoded as a matrix \[
\begin{pmatrix}
 M_1
  & \rvline & M_2 \\
\hline
  M_3 & \rvline &
 M_4
\end{pmatrix}
\]

For practical purposes, only the upper right block and the lower right block are needed and a biquandle is sometimes represented as follows. 

\[
\begin{pmatrix}
 M_2
  & \rvline & M_4 
\end{pmatrix}
\]

For example, consider the following matrix representing a biquandle. The left block corresponds to $\overline{\triangleright}$ and the right block corresponds to $\underline{\triangleright}$. To see what $x\underline{\triangleright} y$ is in the matrix, we go to the $x$th row and the $y$th column of the right block. For instance, $2\underline{\triangleright} 1 = 3.$

\[
\begin{pmatrix}
1
  & 1 & 1 & \rvline & 1 & 1 & 1 \\

  3 & 2 & 2 & \rvline & 3 & 2 & 2

 \\

  2 & 3 & 3 & \rvline & 
2 & 3 & 3
\end{pmatrix}
\]

Biquandles have been used to distinguish knots before, but to the authors' knowledge, it provides a new lower bound for $b_2(K),$ where $K$ is a virtual knot. We demonstrate how this works now.

\begin{definition}
    A \textit{short arc} is a pair of consecutive numbers in the Gauss code. In a diagrammatic realization, a \textit{short arc} is a segment between two classical crossings.
\end{definition}

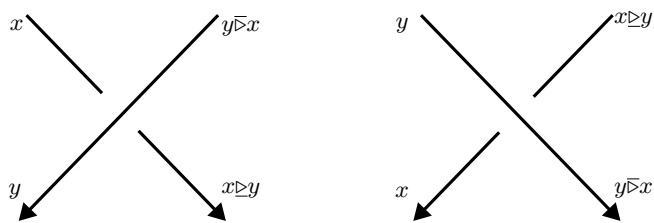
\begin{figure}[h]\label{fig:biqrules}
\labellist
\small\hair 2pt
\pinlabel $x$ at -1 106
\pinlabel $y\overline{\triangleright}x$ at 120 106
\pinlabel $x\underline{\triangleright}y$ at 120 17
\pinlabel $y$ at -2 16

\pinlabel $y$ at 207 106
\pinlabel $y\overline{\triangleright}x$ at 332 18
\pinlabel $x\underline{\triangleright}y$ at 332 109
\pinlabel $x$ at 207 16

\endlabellist
   \centering
\includegraphics[width=8cm]{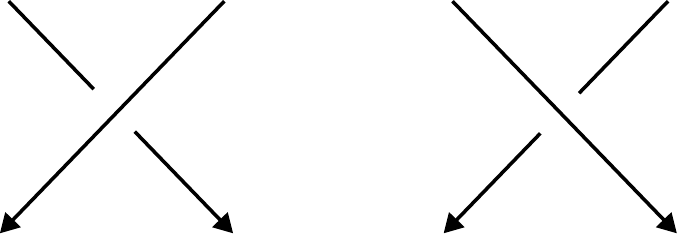}
\caption{For a coloring to be legitimate, these relations have to be satisfied at every crossing.}
\end{figure}

In Section \ref{sec:review}, we assign an entire strand a color. In this section, we will label the short arc of diagrams with elements from the biquandle. The number of legitimate labelings will be a lower bound of $b_2.$ In this paper, we use Sam Nelson's Python code, we can compute $col_{X}(K)$. 

\begin{definition}
    A labeling by a biquandle is \textit{legitimate} (sometimes also referred to as \textit{consistent}) if the relationship in Figure \ref{fig:biqrules} is satisfied at every crossing.
\end{definition}

In \cite{nelson2006matrices}, the authors enumerated biquandles up to order $n$. In that same paper (and in various other papers on biquandles), it is shown that the number of legitimate colorings is an invariant of virtual knots. That is, any allowable moves that do not change the knot type of a virtual knot will not alter the number of colorings.

It was shown before in \cite{clark2014quandle} that the number of quandle colorings $col_{X}(K)$ by a quandle $X$ of order $n$ estimates the bridge number. That is $n^{b_1(K)}\geq col_{X}(K)$, where $b_1(K)$ denotes the bridge number of the knot type. This argument can easily be generalized to biquandle, so that it lower bounds $b_2$.

\begin{proposition}
    Let $X$ be a biquandle of order $n$, then $n^{b_2(K)}\geq col_{X}(K)$
\end{proposition}

\begin{proof}
    Let $D$ be a diagram realizing the bridge number. This means $D$ has $b_2(K)$ local maxima with respect to the standard height function $h:\mathbb{R}^2\rightarrow \mathbb{R}$. We can further assume that all maxima are above all the crossings (see the top left of Figure \ref{fig:braid}). Any coloring on the strands containing these maxima extend to the whole diagram since removing the maxima from $D$ gives a collection of strings that only move downward and never backtrack. These strings can be further divided into blocks such that each block contains exactly one crossing (See the bottom of Figure \ref{fig:braid}). Now, it is quite clear that any labeling at the top extend to the bottom via the biquandle rules. More precisely, consider a block that has $k-2$ strands going straight down and exactly one crossing caused by the $(k-1)$-th and the $k$-th strand. Suppose that the biquandle labels at the top of the block are $a_1,...,a_{k}$. Then, the labels at the bottom of the block are $a_1,...,a_{k-1}\overline{\triangleright} a_k$ or $a_1,...,a_{k-1}\underline{\triangleright} a_k$.

    To form a knot, the bottom of the strings have to connect up, forming the local minima. These connections may or may not give consistent colorings of the whole diagram. If all the connections give consistent coloring, then $n^{b_2(K)}= col_{X}(K)$. However, in general, we just have an inequality $n^{b_2(K)}\geq col_{X}(K)$.
\end{proof}

\begin{figure}[h]\label{fig:braid}
   \centering
\includegraphics[width=8cm]{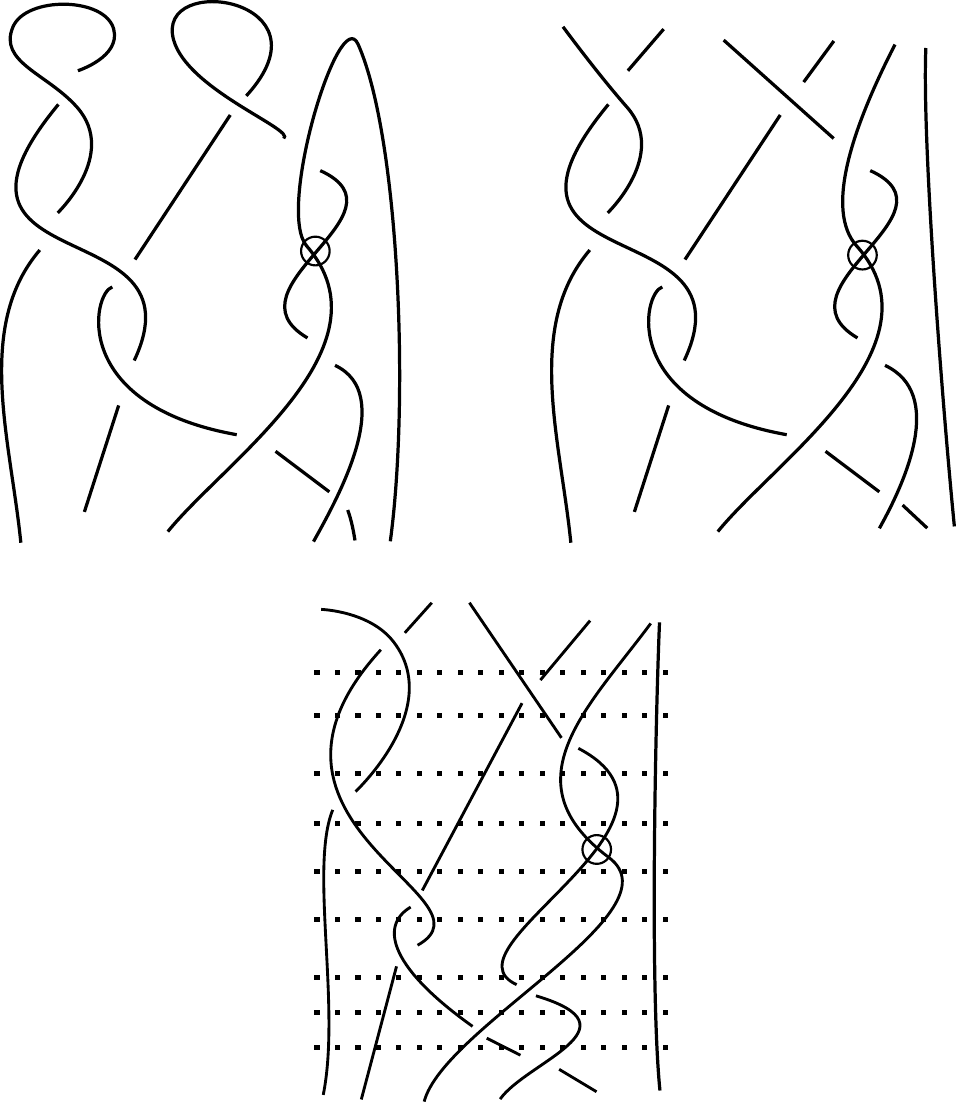}
\caption{(Top left) A position where all local maxima are above all the crossings}
\end{figure}
We now set up the notations for the next theorem. In Figure \ref{fig:difbridge} along the vertical direction, we see many copies tied together of a virtual knot that looks like a bow-tie with 2 virtual crossings and 4 real crossings. Each such bow-tie diagram is denoted as \textit{Kish} (these bow-tie shaped knots appeared in the literature as Kishino knots as they exhibit interesting behaviors). We also denote the result of summing $n$ copies of \textit{Kish} as $Kish_n.$

\begin{theorem}
    There exists a family of virtual knots $\{K_{m,n}\}$ such that $b_1(K_{m,n}) = m$ and $\lim_{n\rightarrow\infty} b_2(K_{m,n})-b_1(K_{m,n})=\infty.$
\end{theorem}
\begin{figure}[h]\label{fig:difbridge}
   \centering
\includegraphics[width=8cm]{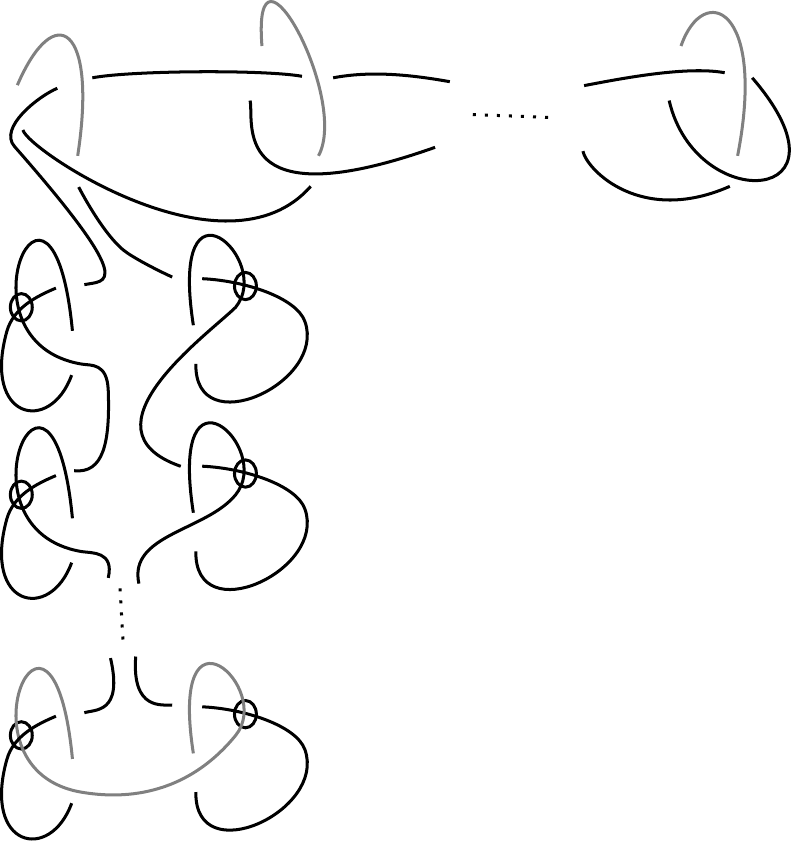}
\caption{The knot $K_{m,n}.$ Here, the vertical dots mean there are $n$ copies of the 4-crossing virtual knots tied together. The horizontal dots mean there are $m$ copies of the 3-crossing classical knots tied together.}
\end{figure}
\begin{proof}
    \textbf{Claim 1}: $b_1(K_{m,n}) = m$ 

   Let $X$ be a quandle that is not a biquandle. Pick any knot $K_m$ such that $Col_X(K_m) = n^{b_1(K_m)}$ and $b_1(K_m)=m.$ Many such knots exist, but for concreteness, we can let $K_m$ be the $(m-1)$-fold connected sum of the trefoil knot. Using the grey strands in Figure \ref{fig:difbridge} as seeds, we see that the knot $K_{m,n}$ has the property that $b_1(K_{m,n}) \leq m.$ To see that $b_1(K_{m,n}) \geq m,$ observe that we can label all arcs of the $n$-fold virtual knot sum (i.e. the vertical part of Figure \ref{fig:difbridge}) with one quandle element. That is, any coloring $f$ of $K_m$ extends to a uniquely determined coloring of $K_{m,n}=Kish_n\# K_m.$ 

        \textbf{Claim 2}: $\lim_{n\rightarrow\infty} b_2(K_{m,n})-b_1(K_{m,n})=\infty.$ 

        Let $Y$ be the biquandle 
        
\[
\begin{array}{cccc|cccc}
1 & 4 & 2 & 3 & 1 & 3 & 4 & 2 \\
2 & 3 & 1 & 4 & 3 & 1 & 2 & 4 \\
4 & 1 & 3 & 2 & 2 & 4 & 3 & 1 \\
3 & 2 & 4 & 1 & 4 & 2 & 1 & 3 \\ 
\hline
1 & 3 & 4 & 2 & 1 & 4 & 2 & 3 \\
3 & 1 & 2 & 4 & 2 & 3 & 1 & 4 \\
2 & 4 & 3 & 1 & 4 & 1 & 3 & 2 \\
4 & 2 & 1 & 3 & 3 & 2 & 4 & 1  
\end{array}.
\]

Claim 2 is verified once we show that $\lim_{n\rightarrow\infty} Col_Y(K_{m,n})=\infty.$ To see this, remove a subarc $a$ of the diagram of $Kish$. Denote the resulting diagram as $Kish\backslash a$. We take two copies of the resulting diagram. Consider the following 16 colorings of $Kish$, and hence of $Kish\backslash a$:
\begin{center}
    (1, 1, 1, 1, 1, 1, 1, 1),\\
 (1, 1, 1, 1, 4, 3, 4, 1),\\
 (1, 3, 4, 4, 1, 3, 2, 2),\\
 (1, 3, 4, 4, 4, 1, 3, 2),\\
 (2, 1, 3, 4, 1, 3, 2, 2),\\
 (2, 1, 3, 4, 4, 1, 3, 2),\\
 (2, 3, 2, 1, 1, 1, 1, 1),\\
 (2, 3, 2, 1, 4, 3, 4, 1),\\
 (3, 1, 2, 2, 2, 3, 1, 4),\\
 (3, 1, 2, 2, 3, 1, 4, 4),\\
 (3, 3, 3, 3, 2, 1, 2, 3),\\
 (3, 3, 3, 3, 3, 3, 3, 3),\\
 (4, 1, 4, 3, 2, 1, 2, 3),\\
 (4, 1, 4, 3, 3, 3, 3, 3),\\
 (4, 3, 1, 2, 2, 3, 1, 4),\\
 (4, 3, 1, 2, 3, 1, 4, 4).\\
\end{center}

Each list above is a labeling on 8 short arcs of $Kish$. For instance, (1,1,1,1,1,1,1,1) is the labeling where all short arcs receive the label 1. Now, observe there are 4 lists with the same number as the first entry. This implies the following. Suppose that the label of the strand containing $a$ before the removal is $x$. Then, in the second copy of $Kish\backslash a$, there are exactly 4 colorings where $Kish\backslash a$ also has the color $x.$

In other words, $Kish$ has 16 legitimate colors by the biquandle $Y$. As we join two copies of $Kish$ to form $Kish_2$, we can notice that for each color of the joining strand, there are 4 colorings for the second copy, giving 16$\times 4=64$ colors. Our result then follows by induction.
\end{proof}

\paragraph{Bounding version 2 bridge numbers for virtual knots:}
Recall that in contrast with classical knots, the two versions of bridge numbers for virtual knots are different.
By using biquandle
\[
(((1, 3, 4, 2), (3, 1, 2, 4), (2, 4, 3, 1), (4, 2, 1, 3)),
             ((1, 4, 2, 3), (2, 3, 1, 4), (4, 1, 3, 2), (3, 2, 4, 1))),
\]
we were able to compute the lower bound for version 2 bridge numbers of virtual knots up to 5 crossings.  
For computing version 1 bridge number, we present it separately in Section \ref{Sec:Experiments}.

\section{Implementation}\label{Sec:Experiments}
\subsection{Data set}\label{SectionInput}
In this section, we explain our implementation for obtaining the dataset of version 1 bridge numbers for classical knots with $16$ crossings and virtual knots with $15$ crossings. It is noteworthy that the method employed here is effective for knots of up to $26$ crossings, as we utilize the Wirtinger number algorithm \cite{Paul} to compute the upper bound of bridge numbers.
Our data and code will be made available at https://github.com/hanhv/Bridges. 
We begin with the following dataset, sourced from \cite{Code}:
\begin{equation}\label{original_data}
\texttt{all\_data\_A.xlsx}, \dots, \texttt{all\_data\_E.xlsx}.    
\end{equation} 

\subsubsection{Classical knots}
There are a total of 1,701,936 distinct classical knots up to 16 crossings, with over 81\% of them (1,388,705 knots) being 16-crossing. Their bridge numbers range between $2$ and $5$ (see \cite{Blair_2020}). 
According to the work of De Wit \cite{de20072} and Blair-Kjuchukova-Morrison \cite{blair2022coxeter}, all 2-bridge classical knots are known.
In the spreadsheets \eqref{original_data}, when the value in the column titled ``Wirtinger Number" is 2 or 3, it is the actual bridge number.
If the value is 4, it only serves as the upper bound.
In this case, if the value in the column titled ``Any Homomorphism? (1=Y,0=N)" is 1, then it represents the exact bridge number.

As explained in Section \ref{section_biquandles}, we can use the quandles method to search for better lower bound for bridge numbers of the knots in the unknown. 
Using the quandle $((1,3,2),(3,2,1),(2,1,3))$, we obtain 122 new knots with 4-bridges. 
Figure \ref{fig:bar_chart} illustrates the distribution of bridge numbers among 16-crossing classical knots. The predominant bridge numbers are 3 and 4: our labeled dataset comprises $892,850$ Gauss codes with bridge number of 3 and $189,688$ Gauss codes with that of bridge number 4.  

\begin{figure} 
    \centering
    \includegraphics[width=0.8\textwidth]{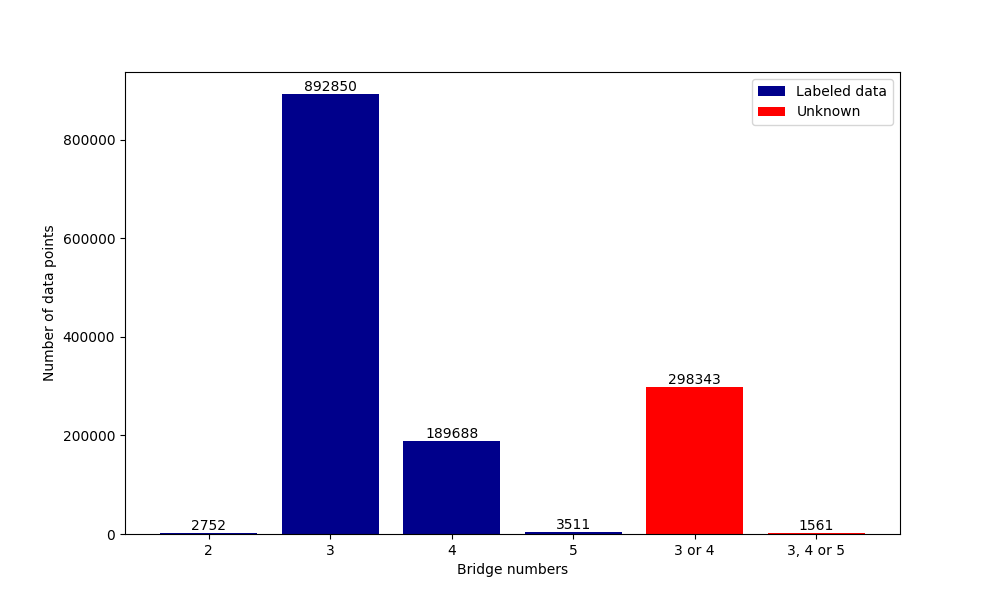}
    \caption{Bridge numbers of 16-crossing knots}
    \label{fig:bar_chart}
\end{figure}
  
\subsubsection{Virtual knots} 
Let $g$ be a Gauss code representing a 16-crossing classical knot.
For each $k\in{1, \dots, 16}$, by removing $k$ and $-k$ from $g$, we obtain a virtual knot with 15 crossings. This effectively replaces one real crossing with a virtual crossing.
By doing so, we obtain over 100 million, not necessarily distinct, virtual knots. We will give some intuition to why they should be mostly distinct in the next subsection.
The bridge numbers range from 1 to 5. 
Note that in contrast with classical knots, where the only 1-bridge classical knot is the unknot, there are infinitely many nontrivial 1-bridge virtual knots (see \cite{byberi2008virtual}). 
To compute the bridge numbers, we use \cite{Paul, Code}. There is only a small modification in the function $\texttt{calc\_wirt\_info}$ when computing the Wirtinger number:
\begin{verbatim}
def calc_wirt_info(knot_dict):
n = 1
# n = 1 instead of 2 as in the original file
# rest of the function implementation
\end{verbatim} 
Among all the obtained virtual knots, we were able to compute the exact bridge numbers for 1,308,309 Gauss codes, with the distribution as in Figure \ref{fig:bar_chart_virtual}. 
Interestingly, among them, only two Gauss codes have a bridge number of 1. 
Based on the acquired data, a natural question arises: does the method used to generate virtual knots induce an imbalance in the dataset, or is it indeed a theoretical phenomenon? 
\begin{figure} 
    \centering
    \includegraphics[width=0.8\textwidth]{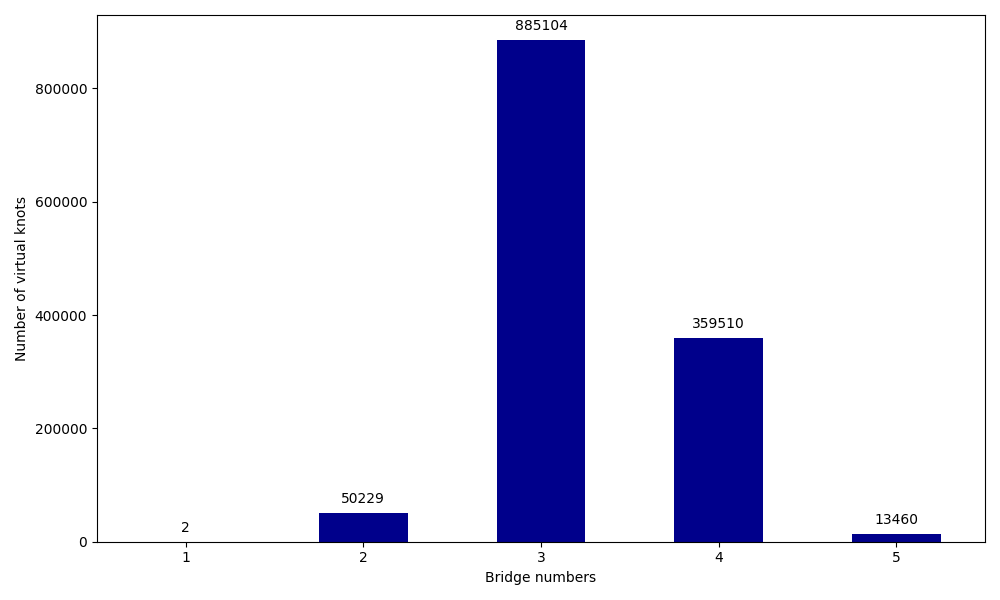}
    \caption{Labelled data for 15-crossing virtual knots}
    \label{fig:bar_chart_virtual}
\end{figure}
In case the exact bridge numbers are unknown, we use quandle $((1, 3, 2), (3, 2, 1), (2, 1, 3))$ to search for better lower bounds.  
\subsection{Behaviors of the Jones polynomial under local moves}

In this paper, we have generated new examples of virtual knots by removing a pair of numbers from Gauss codes for classical knots of 16 crossings. It is possible that doing such a process to two different Gauss codes yields equivalent virtual knots. Although we did not put in the work of verifying that all of our 15 crossing virtual knots are distinct, we can see why the resulting virtual knots are distinct most of the time by considering \cite{kamada2002virtualized}.

More specifically, suppose that $K$ is a 16-crossing classical knot. Let $K'$ be the result of replacing the number $\pm j$ with $\mp j$. Let $K_v$ be the result of removing a pair of numbers (say, 1,-1) from the Gauss code of $K$. Theorem 1 of \cite{kamada2002virtualized} states that the Jones polynomial $V_{K_v}(A)$ of $K_v$ is

\begin{align*}
    \frac{A^{\pm 3}V_K(A)+A^{\mp 3}V_{K'}(A)}{A^3+A^{-3}}
\end{align*}

This means that if we perform the removal of a pair of numbers to two inequivalent Gauss codes $K_1$ and $K_2$ such that $K_1$ and $K_2$ have distinct Jones polynomial, then the resulting 15-crossing virtual knots will have distinct Jones polynomial, and hence distinct virtual knots. The instances where two distinct classical knots share identical Jones polynomial can happen, but not as often. Knots with the same Jones polynomial are not classified. Examples were construct using a local move called \textit{mutation}. Of course, removing a pair of numbers from Gauss codes of two distinct classical knots with the same Jones polynomial may still result in distinct virtual knots, but other invariants are needed to distinguish them.

\subsection{Experiment with machine learning}

In this section, we assess the effectiveness of machine learning methods in classifying the bridge numbers of knots. For our experiment, we specifically examine the set of 16-crossing classical knots, with bridge numbers either 3 or 4. Similar experiments can be done with virtual knots.

Machine learning classifiers encompass various algorithms for pattern recognition and prediction. Traditional methods like Random Forest, Decision Trees, KNN, SVM, and Logistic Regression rely on statistical principles for structured datasets, offering interpretability and ease of implementation. In contrast, neural networks exhibit complexity and flexibility, excelling in handling unstructured and high-dimensional data, often achieving state-of-the-art performance. Hyperparameter tuning aims to enhance model performance through techniques such as grid search, random search, and gradient-based optimization.

To evaluate the performance of classification models, in this experiment, we consider the test accuracy and F1 score. Accuracy measures the proportion of correctly classified instances out of the total instances. However, accuracy can be misleading, especially when dealing with imbalanced datasets. F1 score, on the other hand, considers both precision and recall, providing a more balanced assessment of a model's performance, particularly in situations where class distribution is uneven or when false positives and false negatives carry different costs.

For model selection, we initially assessed our original labeled datasets using various classifiers. Note that hyperparameter tuning was not conducted at this stage. Instead, default hyperparameters were used for all classifiers, except for the Logistic Regression and Neural Network models as they did not converge with the default settings:
\[
\text{Logistic Regression}: \quad \texttt{max\_iter=10000}
\]
and 
\[
\text{Neural Network}: \quad \texttt{max\_iter=500, hidden\_layer\_sizes=(100, 100, 100))}.
\] 
The performance metrics (test accuracy and weighted average F1-score) for these models can be found in Table \ref{tab:model_performance}.
Based on the evaluation, the Random Forest model stands out as a promising choice due to its consistently high performance. 
%cross-validation
For assessing the Random Forest model's robustness, we check the cross-validation scores, which are presented in Table \ref{tab:cross-validation}.
%sampling
To address imbalanced datasets, methods such as undersampling and oversampling are employed, with common oversampling techniques including SMOTE and ADASYN. In our dataset, where cyclic rotations of a Gauss code signify the same knot, we can also utilize this approach for oversampling to maintain data integrity while achieving balance. Table \ref{tab:performanceRF} displays the performance of Random Forest with various sampling methods. 
Figures \ref{fig:confusion_matrixRF} and \ref{fig:roc_curveRF} depict the Confusion matrix and ROC curve for Random Forest using the original labeled dataset, split with 80\% training and 20\% testing data. Utilizing the Random Forest model with the original labeled dataset for further tuning and prediction is recommended.

% \begin{table}[htbp]
% \centering
% \caption{Model performances with original labeled dataset: 80\% training, 20\% testing split}
% \medskip
% \renewcommand{\arraystretch}{1.5}
% \label{tab:model_performance}
% \begin{tabular}{@{}lllll@{}}
% \toprule
% \textbf{Model} & \textbf{Training Accuracy} & \textbf{Test Accuracy} & \textbf{F1-score} \\ \midrule
% Random Forest & 1.0 & 0.9568 & 0.9551 \\
% Decision Tree & 1.0 & 0.9414 & 0.9414 \\
% Extra Trees & 1.0 & 0.9398 & 0.9375 \\
% Stochastic Gradient Descent & 0.8247 & 0.8252 & 0.7462 \\
% Gradient Boosting & 0.8707 & 0.8707 & 0.8547 \\
% XGBoost & 0.9307 & 0.9273 & 0.9238 \\
% Gaussian Naive Bayes & 0.6504 & 0.6502 & 0.6867 \\
% Support Vector Machine & 0.8701 & 0.8678 & 0.8542 \\ 
% K-Nearest Neighbors & 0.9284 & 0.8934 & 0.8870 \\ 
% Logistic Regression* & 0.8247 & 0.8252 & 0.7462 \\
% Neural Network* & 0.9193 & 0.8998 & 0.8951 \\
% \bottomrule
% \end{tabular}
% \end{table} 

\begin{table}[htbp]
\centering
\caption{Model performances with original labeled dataset: 80\% training, 20\% testing split}
\medskip
\renewcommand{\arraystretch}{1.5}
\label{tab:model_performance}
\begin{tabular}{@{}lllll@{}}
\toprule
\textbf{Model} & \textbf{Training accuracy} & \textbf{Test accuracy} & \textbf{F1-score} \\ \midrule
        Random Forest & 0.9999 & 0.9570 & 0.9554 \\
        Decision Tree & 1.0000 & 0.9419 & 0.9419 \\
        Extra Trees & 1.0000 & 0.9407 & 0.9385 \\
        % Logistic Regression (default) & 0.8247 & 0.8252 & 0.7462 \\
        Logistic Regression* & 0.8247 & 0.8252 & 0.7462 \\
        Stochastic Gradient Descent & 0.8247 & 0.8252 & 0.7462 \\
        Gradient Boosting & 0.8707 & 0.8707 & 0.8547 \\
        XGBoost & 0.9307 & 0.9273 & 0.9238 \\
        Gaussian Naive Bayes & 0.6504 & 0.6502 & 0.6867 \\
        % Neural Network (default) & 0.8837 & 0.8819 & 0.8705 \\
        Neural Network* & 0.9181 & 0.8998 & 0.8933 \\
        Support Vector Machine & 0.8701 & 0.8678 & 0.8542 \\
        K-Nearest Neighbors & 0.9284 & 0.8934 & 0.8870 \\
        \bottomrule
    \end{tabular}
\end{table}

\begin{table}[htbp]
     \centering
    \caption{Random Forest: cross-validation scores (10 folds)}
    \medskip
    \renewcommand{\arraystretch}{1.5}
    \label{tab:cross-validation}
    \begin{tabular}{|l|l|c|}
    \hline
    \textbf{Metric} & \textbf{Mean} & \textbf{Standard deviation} \\ \hline
    Accuracy & 0.9599 & 0.0005 \\\hline
    F1-score & 0.9582 & 0.0005 \\\hline
    ROC-AUC & 0.9889 & 0.0003 \\\hline 
    \end{tabular}
\end{table}

\begin{table}[htbp]
\caption{Random Forest: different sampling method for training set}
\medskip
\renewcommand{\arraystretch}{1.5}
\centering
\begin{tabular} {|l|l|l|l|}
\hline
\textbf{Sampling technique} & \textbf{Training accuracy} & \textbf{Test accuracy} & \textbf{F1-score} \\ \hline
Original Data               & 1.0                        & 0.9579    & 0.9564 \\ \hline
Undersample                 & 0.9484        & 0.9221     & 0.9259 \\ \hline
SMOTE                       & 1.0                        & 0.9553     & 0.9544 \\ \hline
Oversample  (cyclic rotations)                 & 1.0                        & 0.9551   & 0.9534\\ \hline
\end{tabular}
\label{tab:performanceRF}
\end{table}

\begin{figure} 
    \centering
    \includegraphics[width=0.7\textwidth]{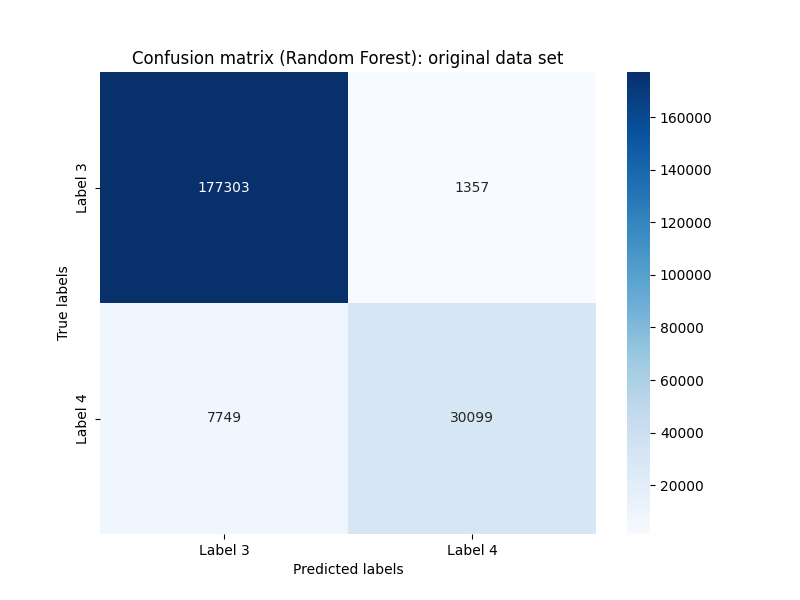}
    \caption{Confusion matrix for Random Forest using the original labeled dataset, split with 80\% training and 20\% testing data}
    \label{fig:confusion_matrixRF}
\end{figure}

\begin{figure} 
    \centering
    \includegraphics[width=0.7\textwidth]{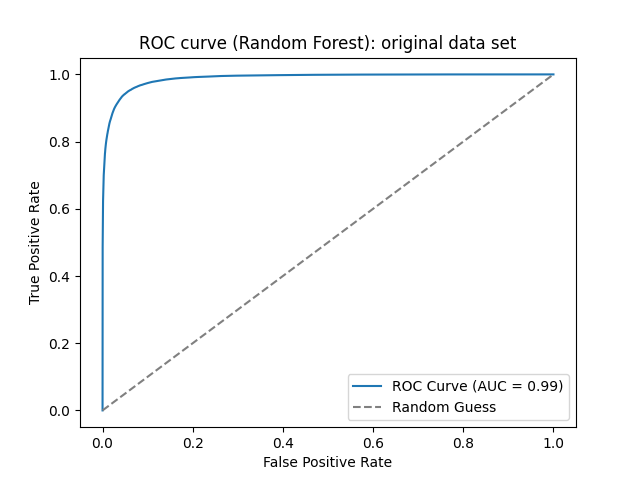}
    \caption{ROC curve for Random Forest using the original labeled dataset, split with 80\% training and 20\% testing data}
    \label{fig:roc_curveRF}
\end{figure}

 \section{Conclusions}\label{sec:Conclusions}
We have provided lower bounds for the second bridge number $b_2(K)$ of virtual knots by biquandles. There is also an upper bound for $b_2(K)$ in terms of the \textit{virtual braid index}, denoted by $braid(K)$. This is the minimum number $r$ needed to present $K$ as the closure of a $r$-strand virtual braid. One of the directions we can pursue is to find an efficient way to estimate the virtual braid index from a Gauss code.

 In conclusion, this study sheds light on the challenges associated with determining bridge numbers for both classical and virtual knots.   Moving forward, future research could focus on refining machine learning algorithms to better accommodate the intricacies of virtual knots and improve classification accuracy. Additionally, exploring alternative computational techniques and datasets could offer valuable insights into enhancing our understanding of knot theory and its applications in various fields.
% \newpage

% \subsection{Open questions}
\bibliographystyle{plain}
\bibliography{sample}
\end{document}